\documentclass[10pt]{amsart}
\usepackage{graphicx, epstopdf, amsthm, amssymb}

\long\def\pythonbegin#1\pythonend{}

\pythonbegin
from __future__ import division  # sage with comment, python no comment
from pyfig import *
import math
beginfile()
\pythonend

\allowdisplaybreaks

\numberwithin{equation}{section}\overfullrule 5pt
\newtheorem{thm}{Theorem}
\newtheorem{cor}[thm]{Corollary}
\newtheorem{lem}[thm]{Lemma}
\newtheorem{prop}[thm]{Proposition}

\theoremstyle{definition}

\title[Tangent and Genocchi numbers]{
Combinatorial proofs of some properties of tangent and Genocchi numbers
}

\date{February 26, 2018}
\author{Guo-Niu HAN}
\author{Jing-Yi Liu$^*$}
\address{I.R.M.A., UMR 7501, Universit\'e de Strasbourg
et CNRS, 7 rue Ren\'e Descartes, 67084 Strasbourg, France}
\email{guoniu.han@unistra.fr}

\address{Department of Mathematics, Beijing Normal University,
Beijing 100875, China}
\email{jingyi.math@gmail.com}

\thanks{*Jing-Yi Liu is the corresponding author.}

\subjclass[2010]{05A15, 05A19, 05C05, 05A10, 11A07}

\keywords{tangent numbers, Genocchi numbers, divisibility, hook length formula
$k$-ary tree}

\begin{document}

\begin{abstract} 
The tangent number $T_{2n+1}$ is equal to the number of
	increasing labelled complete binary trees with $2n+1$ vertices.
This combinatorial interpretation immediately proves that $T_{2n+1}$ is divisible by $2^n$.
However, a stronger divisibility property is known in the studies
of Bernoulli and Genocchi numbers, namely, the divisibility of
$(n+1)T_{2n+1}$ by $2^{2n}$. The traditional proofs of this fact
need significant calculations. In the present paper,
we provide a combinatorial proof of the latter divisibility by
using the hook length formula for trees.
Furthermore, our method is extended to $k$-ary trees, leading to a
new generalization of the Genocchi numbers.
\end{abstract}
\maketitle

\bigskip

\section{Introduction} \label{sec:introduction} 
The {\it tangent numbers} 
\footnote{Some mathematical literature uses a slightly different notation where $\tan x$ is written $T_1 x + T_2 x^3/3! + T_3 x^5/5! + \cdots$ 
(See \cite{KnuthBuckholtz1967})}
$(T_{2n+1})_{n\geq 0}$
appear in the Taylor expansion of $\tan(x)$:
\begin{equation}
	\tan x = \sum_{n\geq 0} T_{2n+1} \frac {x^{2n+1}}{(2n+1)!}.
\end{equation}
It is known that the tangent number $T_{2n+1}$ is equal to the number of
all {\it alternating permutations} of length $2n+1$  
(see \cite{Andre1879, Euler1755, KnuthBuckholtz1967, Nielsen1923}).
Also, $T_{2n+1}$ counts the number of
{\it increasing labelled complete binary trees}  with $2n+1$ vertices.
This combinatorial interpretation immediately implies 
that $T_{2n+1}$ is divisible by $2^n$.
However, a stronger divisibility property is known related to the study of Bernoulli and Genocchi numbers \cite{Carlitz1960,Carlitz1971, RiordanStein1973}, as stated in the following theorem.
\begin{thm}\label{th:tan}
The number $(n+1)T_{2n+1}$ is divisible by $2^{2n}$, and the quotient
is an odd number.
\end{thm}
The quotient is called {\it  Genocchi number} and denoted by 
\begin{equation}\label{eq:genocchi}
G_{2n+2}:=(n+1)T_{2n+1}/2^{2n}.
\end{equation}
Let
$$g(x):=\displaystyle\sum_{n\ge 0}G_{2n+2}\frac{x^{2n+2}}{(2n+2)!}$$
be the exponential generating function for the Genocchi numbers. Then, \eqref{eq:genocchi} is
equivalent to
\begin{equation}\label{eq:gx}
g(x)=x\tan{\frac x2}.
\end{equation}

The initial values of the tangent and Genocchi numbers are listed below:
$$
\begin{tabular}{c | c c c c c c c }
	$n$ & 0 & 1 & 2 & 3 & 4 & 5 & 6 \\
	\hline
	$T_{2n+1}$ & 1 & 2 & 16 & 272 & 7936 & 353792 & 22368256\\
	$G_{2n+2}$ & 1 & 1 & 3 & 17 & 155 & 2073 & 38227\\
	\hline
\end{tabular}
$$

\medskip

The fact that the Genocchi numbers are odd integers
is traditionally proved by using the von Staudt-Clausen theorem on Bernoulli numbers and the little Fermat theorem 
\cite{Carlitz1960,Carlitz1971, RiordanStein1973}.
Barsky \cite{Barsky1980, FoataHan2008} gave a different proof by using the Laplace transform.
To the best of the authors' knowledge, no simple combinatorial proof has been derived yet and it is the purpose of this paper to provide one.
Our approach is based on the geometry of the
so-called
{\it leaf-labelled tree}
and the fact that the hook length $h_v$ of such a tree is always an odd integer
(see Sections \ref{sec:BT} and~\ref{sec:tan}).

In Section \ref{sec:kary}
we consider the $k$-ary trees
instead of the binary trees
and
obtain a new generalization of the Genocchi numbers.
For each integer $k\geq 2$, let $L_{kn+1}^{(k)}$ be the number of increasing labelled complete $k$-ary trees with $kn+1$ vertices. 
Thus, $L^{(k)}_{kn+1}$ will appear to be a natural generalization 
of the tangent number.
The general result is stated next.
\begin{thm}\label{th:kary}
	(a) For each integer $k\geq 2$, the integer
	$$\frac{(k^2 n-kn+k)!\,L^{(k)}_{kn+1}}{ (kn+1)!}$$
is divisible by $(k!)^{kn+1}$.

	(b) Moreover, the quotient
\begin{align*}
	M^{(k)}_{k^2 n-kn+k}:=
	\frac{(k^2 n-kn+k)!\, L^{(k)}_{kn+1}}{(k!)^{kn+1}(kn+1)!}\equiv
\begin{cases}
	1\pmod {k}, &k=p, \\
	1\pmod {p^2}, &k=p^t,\ t\ge 2, \\
	0\pmod {k}, &\text{otherwise},
\end{cases}
\end{align*}
where $n\ge 1$ and $p$ is a prime number.
\end{thm}

We can
 realize that Theorem \ref{th:kary} is a direct generalization of Theorem \ref{th:tan},
if we restate the problem in terms of generating functions.
Let $\phi^{(k)}(x)$ and $\psi^{(k)}(x)$ denote the exponential generating functions for 
$L^{(k)}_{kn+1}$ and $M^{(k)}_{k^2n-kn+k}$, respectively, that is,
\begin{align*}
	\phi^{(k)}(x)&=\sum_{n\ge 0}L^{(k)}_{kn+1}\frac{x^{kn+1}}{(kn+1)!}; \\
	\psi^{(k)}(x)&=\sum_{n\ge 0}M^{(k)}_{k^2n-kn+k}\frac{x^{k^2n-kn+k}}{(k^2n-kn+k)!}.
\end{align*}
If $k$ is clear from the context, the superscript $(k)$ will be omitted.
Thus, we will write  $L_{kn+1}:=L^{(k)}_{kn+1},\,
	M_{k^2 n-kn+k} := M^{(k)}_{k^2 n-kn+k},\,
\phi(x):=\phi^{(k)}(x),\, 
\psi(x):=\phi^{(k)}(x)$.
From Theorem \ref{th:kary} we have
\begin{align*}
\phi'(x)=1+\phi^k(x);
\end{align*}
\begin{align*}
{\psi(x)}=x \cdot \phi\left(\displaystyle \frac{x^{k-1}}{k!}\right).
\end{align*}
The last relation becomes the well-known formula \eqref{eq:gx} when $k=2$.

Several generalizations of the Genocchi numbers
have been studied in recent decades.
They are based on
 the Gandhi polynomials \cite{Domaratzki2004,Carlitz1971, RiordanStein1973},
Seidel triangles \cite{DumontRand1994, ZengZhou2006},
continued fractions \cite{Viennot1982, HanZeng1999den}, combinatorial models \cite{HanZeng1999den}, etc.
Our generalization seems to be the first extension
dealing with
 the divisibility of
$(n+1)T_{2n+1}$ by $2^{2n}$. It also raises the following open problems.
\smallskip

{\bf Problem 1}. Find a proof of Theorem \ref{th:kary} \`a la Carlitz, or \`a la Barsky.
\smallskip

{\bf Problem 2}. Find the Gandhi polynomials, Seidel triangles, continued fractions and a combinatorial model for the new generalization of Genocchi numbers $M_{k^2n-kn+k}$ \`a la Dumont.
\smallskip

{\bf Problem 3}. Evaluate $m_n:=M_{k^2n-kn+k} \pmod k$ for $k=p^t$, where $p$ is a prime number and $t\geq 3$. It seems that the sequence $(m_n)_{n\geq 0}$ is always periodic for any $p$ and $t$. 
Computer calculation has provided the initial values:
\begin{align*}
	(m_n)_{n\geq 0} &= (1,1,5,5,1,1,5,5,\cdots) \qquad \text{for } k=2^3,\\
 	(m_n)_{n\geq 0} &= (1,1,10,1,1,10,1,1,10\cdots) \qquad \text{for } k=3^3,\\
    (m_n)_{n\geq 0} &= (1,1,126,376,126,1,1,126,376,126,\cdots) \qquad \text{for } k=5^4, \\
    (m_n)_{n\geq 0} &= (1,1,13,5,9,9,5,13,1,1,13,5,9,9,5,13,\cdots) \qquad \text{for } k=2^4. 
\end{align*}


\section{Increasing labelled binary trees}\label{sec:BT} 
In this section we recall some basic notions on 
increasing labelled binary trees.
Consider the set $\mathcal{T}(n)$ of all (unlabelled) binary trees with $n$ vertices.
For each $t\in \mathcal{T}(n)$
let $\mathcal{L}(t)$ denote the set of all
{\it increasing labelled binary trees} of shape $t$, 
obtained from $t$ by labeling its $n$ vertices with $\{1,2,\ldots,n\}$ in such a way
that the label of each vertex is less than that of its descendants.
For each vertex~$v$ of~$t$, the {\it hook length} of~$v$,
denoted by $h_v(t)$ or $h_v$, is
the number of descendants of~$v$ (including $v$).
The {\it hook length formula} (\cite[\S5.1.4. Ex. 20]{Knuth1998Vol3})
claims that
the number of increasing labelled binary trees of shape $t$
is equal to $n!$ divided by the product of the $h_v$'s ($v\in t$)
\begin{equation}\label{eq:hooklength}
	\#\mathcal{L}(t)=\frac{n!}{\prod_{v\in t} h_v}.
\end{equation}

Let $\mathcal{S}(2n+1)$ denote
the set of all
{\it complete binary trees} $s$ with $2n+1$ vertices, which are defined to be
the binary trees such that the two
subtrees of each vertex are, either both empty, or both non-empty.
For example, there are five complete binary trees with $2n+1=7$ vertices,
labelled by their hook lengths in Fig.~1.

\medskip

\pythonbegin
beginfig(1, "1.5mm");
setLegoUnit([3,3])
#showgrid([0,0], [20,14])  # show grid
r=0.15
rtext=r+0.5  # distance between text and the point

def ShowPoint(ptL, labelL, dir, fill=True):
	[circle(p[z-1], r, fill=fill) for z in ptL]
	[label(p[ptL[z]-1], labelL[z], dist=[rtext, rtext], dist_direction=dir) for z in range(len(ptL))]

dist=[1,1,1,1]
p=btree([6,4,7,2,5,1,3], [4,8], dist=dist, dot="fill", dotradius=r)
ShowPoint([6,7,5,3], [1,1,1,1], 270)
ShowPoint([4,2,1], [3,5,7],135)
label(addpt(p[0],[0,1.6]), "$s_1$")

p=btree([4,2,6,5,7,1,3], [8.2,8], dist=dist, dot="fill", dotradius=r)
ShowPoint([4,6,7], [1,1,1], 270)
ShowPoint([2,1], [5,7],135)
ShowPoint([5,3], [3,1],45)
label(addpt(p[0],[0,1.6]), "$s_2$")

p=btree([2,1,6,4,7,3,5], [12.4,8], dist=dist, dot="fill", dotradius=r)
ShowPoint([5,6,7], [1,1,1], 270)
ShowPoint([4,2,1], [3,1,7],135)
ShowPoint([3], [5],45)
label(addpt(p[0],[0,1.6]), "$s_3$")

p=btree([2,1,4,3,6,5,7], [16.6,8], dist=dist, dot="fill", dotradius=r)
ShowPoint([2,4,6,7], [1,1,1,1], 270)
ShowPoint([5,3,1], [3,5,7],45)
label(addpt(p[0],[0,1.6]), "$s_4$")

p=btree([4,2,5,1,6,3,7], [22.8,8], dist=[1.2, 0.8], dot="fill", dotradius=r)
ShowPoint([5,4,6,7], [1,1,1,1], 270)
ShowPoint([1,2], [7,3],135)
ShowPoint([3], [3],45)
label(addpt(p[0],[0,1.6]), "$s_5$")

endfig();
\pythonend
\begin{center}{\includegraphics[width=0.8\textwidth]{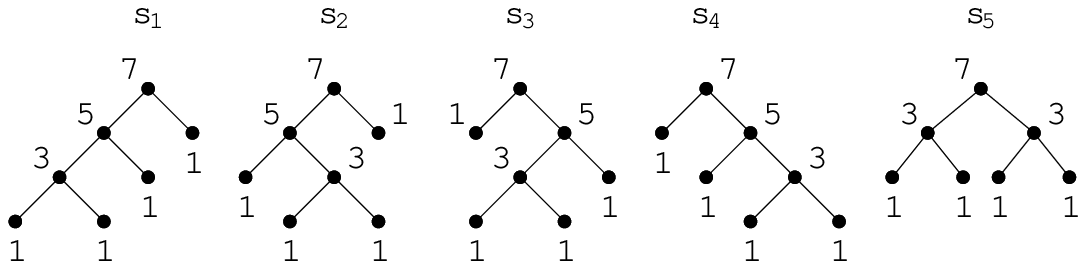}}\end{center}
\begin{center}{Fig.~1.~Complete binary trees with 7 vertices}\end{center}

\medskip
We now define an equivalence relation on $\mathcal{S}(2n+1)$, called {\it pivoting}.
A {\it basic pivoting}  is an exchange of
the two subtrees of a non-leaf vertex $v$.
For $s_1, s_2\in \mathcal{S}(2n+1)$, if $s_1$ can be changed to $s_2$ by a finite sequence of basic pivotings, we write $s_1\sim s_2$.
It's routine to check that $\sim$ is an equivalence relation. Let $\mathcal{\bar S}(2n+1) = \mathcal{S}(2n+1)/\!\!\sim$.
Since $s_1\sim s_2$ implies that $\#\mathcal{L}(s_1)=\#\mathcal{L}(s_2)$,
we define $\#\mathcal{L}(\bar s)=\#\mathcal{L}(s)$ for $s\in \bar s$.
Then
\begin{equation}\label{eq:normaltree}
T_{2n+1} = \sum_{\bar s\in\mathcal{\bar S}(2n+1)} T(\bar s),
\end{equation}
where
\begin{equation}\label{eq:Ts}
T(\bar s)=\sum_{s \in \bar s} \# \mathcal{L}(s) =
\#\bar s \times \#\mathcal{L}(\bar s).
\end{equation}
For example, consider $\mathcal{S}(7)$ (see Fig. 1), we have
\[\begin{array}{cccccc}
\text{shape} & s_1 & s_2 & s_3 & s_4 & s_5 \\
\prod_v h_v & 3\cdot 5\cdot 7 & 3\cdot 5\cdot 7 & 3\cdot 5\cdot 7 & 3\cdot 5\cdot 7 & 3\cdot 3\cdot 7 \\
n!/\prod_v h_v & 48 & 48 & 48 & 48 & 80 \end{array}\]
Trees $s_1, s_2, s_3$ and $s_4$ belong to the same equivalence class $\overline {s_1}$, while $s_5$ is in another equivalence class
$\overline {s_5}$.
Thus $T(\overline{s_1})=4\times 48=192$, $T(\overline{s_5})=80$ and
$T_7=T(\overline{s_1})+T(\overline{s_5})=272$.

\medskip

The pivoting can also be viewed as an  equivalence relation on
the set $\cup_{s\in \bar s} \mathcal{L}(s)$, that is,
all increasing labelled trees of shape $s$ with $s\in \bar s$.
Since the number of non-leaf vertices
is $n$ in $s$,
there are exactly $2^n$ labelled trees in each equivalence class.
Hence, $T(\bar s)$ is divisible by $2^n$.
Take again the example above, $T(\overline{s_1})/2^3=24$, $T(\overline{s_5})/2^3=10$,
and $T_7/2^3 = 24+10=34$.

\medskip

This is not enough to derive that $2^{2n}\mid (n+1)T_{2n+1}$. However, the above process
leads us to reconsider the question in each equivalence class.
We can show that the divisibility actually holds in
each $\bar s$, as stated below.

\medskip

\begin{prop}\label{th:divisibilitybar}
	For each $\bar s\in \mathcal{S}(2n+1)$, the integer
$(n+1)T(\bar s)$ is divisible by $2^{2n}$.
\end{prop}

\medskip

Let $G(\bar s):= (n+1)T(\bar s)/2^{2n}$. Proposition \ref{th:divisibilitybar} implies that $G(\bar s)$ is an integer.
By \eqref{eq:genocchi} and \eqref{eq:normaltree},
\begin{equation}\label{eq:Genocchi}
G_{2n+2} = \sum_{\bar s\in\mathcal{\bar S}(2n+1)} G(\bar s).
\end{equation}
We give an example here and present the proof in the next section.

For $n=4$, there are three equivalence classes.

\medskip

\pythonbegin
beginfig(2, "1.6mm");
setLegoUnit([3,3])
#showgrid([0,0], [20,14])  # show grid
r=0.15
rtext=r+0.5  # distance between text and the point

def ShowPoint(ptL, labelL, dir, fill=True):
	[circle(p[z-1], r, fill=fill) for z in ptL]
	[label(p[ptL[z]-1], labelL[z], dist=[rtext, rtext], dist_direction=dir) for z in range(len(ptL))]

dist=[1,1,1,1]
p=btree([8,6,9,4,7,2,5,1,3], [4,8], dist=dist, dot="fill", dotradius=r)
ShowPoint([8,9,7,5,3], [1,1,1,1,1], 270)
ShowPoint([6,4,2,1], [3,5,7,9],135)
label(addpt(p[0],[0,1.6]), "$s_1\in\overline{s_1}$")

dist=[1.6,1.6,1,1]
p=btree([6,4,7,2,8,5,9,1,3], [11,8], dist=dist, dot="fill", dotradius=r)
ShowPoint([8,9,7,6,3], [1,1,1,1,1], 270)
ShowPoint([4,2,1], [3,7,9],135)
ShowPoint([5], [3],45)
label(addpt(p[0],[0,1.6]), "$s_2\in\overline{s_2}$")

dist=[1.6,1,1,1]
p=btree([6,4,7,2,5,1, 8,3,9], [18,8], dist=dist, dot="fill", dotradius=r)
ShowPoint([8,9,7,6,5], [1,1,1,1,1], 270)
ShowPoint([4,2,1], [3,5,9],135)
ShowPoint([3], [3],45)
label(addpt(p[0],[0,1.6]), "$s_3\in\overline{s_3}$")

endfig();
\pythonend
\begin{center}{\includegraphics[width=0.8\textwidth]{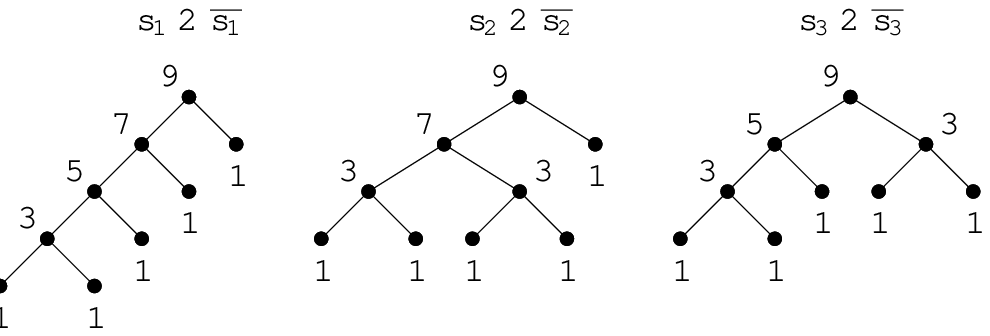}}\end{center}
\begin{center}{Fig.~2.~Three equivalence classes for $n=4$}\end{center}

	\goodbreak

	In this case, Proposition \ref{th:divisibilitybar} and relation \eqref{eq:Genocchi} can be verified by
	the following table.

\nobreak

\begin{center}\begin{tabular}{cccccc}
$\bar s$ & $\#\bar s$ & $\prod h_v$ & $\#\mathcal{L}(\bar s)$ & $T(\bar s)$ & $G(\bar s)$ \\
\hline
	$\overline{s_1}$ & 8 & $3\cdot 5\cdot 7\cdot 9$ & 384 & 3072 & 60 \\
	$\overline{s_2}$ & 2 & $3\cdot 3\cdot 7\cdot 9$ & 640 & 1280 & 25 \\
	$\overline{s_3}$ & 4 & $3\cdot 3\cdot 5\cdot 9$ & 896 & 3584 & 70 \\
\hline
sum & 14 &\quad & \quad & 7936 & 155
\end{tabular}\end{center}

\section{Combinatorial proof of Theorem \ref{th:tan}} \label{sec:tan} 
Let $n$ be a nonnegative integer and
$\bar s\in \mathcal{\bar S}(2n+1)$ be an equivalence class in the set of
increasing labelled complete binary trees.
The key of the proof is the fact that
the hook length $h_v$ is always an odd integer.
For each complete binary tree $s$,
we denote the product of all hook lengths by $H(s)=\prod_{v\in s} h_v$.
Also, let $H(\bar s)=H(s)$ for $s\in \bar s$, since all trees in the
equivalence class $\bar s$ share the same product of all hook lengths.

\begin{lem}\label{lem:hooklength}
For each complete binary tree $s$,
the product of all hook lengths
$H(s)$ is an odd integer.
\end{lem}

By Lemma \ref{lem:hooklength},
Proposition \ref{th:divisibilitybar} has the following equivalent form.

\begin{prop}\label{th:divisibilitybarh}
	For each $\bar s\in \mathcal{\bar S}(2n+1)$, the integer
	$(2n+2)H(\bar s)T(\bar s)$ is divisible by $2^{2n}$.
\end{prop}

\begin{proof}
	By identities \eqref{eq:Ts} and \eqref{eq:hooklength} we have
\begin{align}
(2n+2)H(\bar s)T(\bar s)&=(2n+2)H(\bar s)\times \#\bar s \times \#\mathcal{L}(\bar s) \nonumber\\
&=(2n+2)\times \#\bar s\times (2n+1)! \nonumber \\
	&=(2n+2)!\times \#\bar s.\label{eq:combinatorialinterpretation}
\end{align}

Suppose that $s$ is a complete binary tree with $2n+1$ vertices,
then $s$ has $n+1$ leaves.
Let $s^+$ be the complete binary tree with $4n+3$ vertices obtained
from $s$ by replacing each leaf of $s$ by the complete binary tree
with 3 vertices. So $s^+$ has $2n+2$ leaves.
Let $\mathcal{L}^+(s^+)$
be the set of all leaf-labelled trees of shape $s^+$, 
obtained from $s^+$ by labeling its $2n+2$ leaves 
 with $\{1,2,\ldots, 2n+2\}$.
It is clear that $\#\mathcal{L}^+(s^+)=(2n+2)!$. By (\ref{eq:combinatorialinterpretation}) we have
	the following
combinatorial interpretation:

\medskip
	{\it
	For each $\bar s\in \mathcal{\bar S}(2n+1)$, the number of all leaf-labelled trees of shape $s^+$ such that $s\in \bar s$
is equal to $(2n+2)H(\bar s)T(\bar s)$.
	}

\medskip

This time we take the pivoting for an equivalence equation
on the set of leaf-labelled trees
	$\cup_{s\in \bar s}\mathcal{L}^+(s^+)$.
Since a leaf-labelled tree $s^+$ has $2n+1$ non-leaf vertices, and each non-trivial sequence of
	pivotings will make a difference on the labels of leaves, every equivalence class contains $2^{2n+1}$ elements.
Hence,
we can conclude
	that $(2n+2)H(\bar s)T(\bar s)$ is divisible by $2^{2n+1}$.
\end{proof}

For example, in Fig.~3, we reproduce a labelled tree with $9$ vertices
	and a leaf-labelled tree with $19$ vertices.
	There are $4$ non-leaf vertices
	in the labelled tree and the $9$ non-leaf vertices
	in the leaf-labelled tree, as indicated by the fat dot symbol ``$\bullet$''.
	Comparing with the traditional combinatorial model,  our method
	increases the number of non-leaf vertices. Consequently, we
	establish a stronger divisibility property.

\medskip

\pythonbegin
beginfig(3, "1.6mm");
setLegoUnit([3,3])
#showgrid([0,0], [20,14])  # show grid
dist=[1.6, 1.6, 1]
r=0.15
rtext=r+0.5  # distance between text and the point

def ShowPoint(ptL, labelL, dir, fill=True):
	[circle(p[z-1], r, fill=fill) for z in ptL]
	[label(p[ptL[z]-1], labelL[z], dist=[rtext, rtext], dist_direction=dir) for z in range(len(ptL))]

p=btree([6,4,7,2,8,5,9,1,3],pt=[7,0], dist=dist, dot="frame", dotradius=r, labeled=False)
ShowPoint([1,2,4], [1,2,4], 135, fill=True)
ShowPoint([6,7,8,9,3], [8,5,7,9,3], 270, fill=False)
ShowPoint([5], [6], 60, fill=True)
label(addpt(p[0],[0,2.2]), "Labeled tree")
label(addpt(p[0],[0,1.4]), "$n=4$ non-leaf vertices")

# third tree is composed by 2 trees, because dist is not equal: (9--4) small
dist=[2.2, 2, 1, 0.5]
pa=[19,0]
p=btree([10,6,11,4,12,7,13,2,14,8,15,5,16,9,17,1,3],pt=pa, dist=dist, dot="frame", dotradius=r)
[circle(p[z-1], r, fill=True) for z in [1,2,3,4,5,6,7,8,9]]
ShowPoint([10,11,12,13,14,15,16,17], [5,8,2,6,1,7,10,3], 270, fill=False)

p=btree([2,1,3], pt=p[2], dist=[dist[3]], dot="frame", dotradius=r)
ShowPoint([2,3], [9,4], 270, fill=False)
label(addpt(pa,[-0.8,2.2]), "Leaf-labelled tree")
label(addpt(pa,[-0.8,1.4]), "$2n+1=9$ non-leaf vertices")

endfig();
\pythonend
\begin{center}{\includegraphics[width=0.8\textwidth]{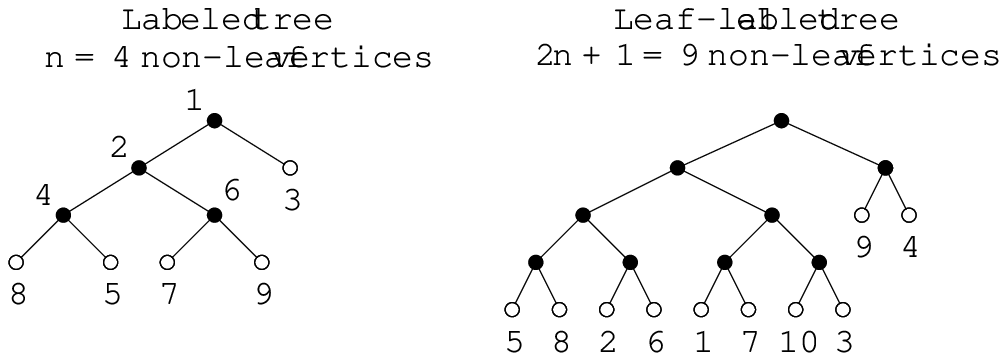}}\end{center}
\begin{center}{Fig.~3.~Trees, non-leaf vertices and divisibilities}\end{center}

\medskip

For proving Theorem \ref{th:tan}, it remains to show that
$G_{2n+2}=\sum G(\bar s)$ is an odd number.
Since $H(\bar s)$ is odd, we need only to
prove that the {\it weighted Genocchi number}
\begin{equation}\label{def:fn}
f(n)=\sum_{\bar s\in \mathcal{\bar S}(2n+1)} H(\bar s)G(\bar s)
\end{equation}
is odd.
For example, in Fig. 2.,
$G_{10}=G(\overline{s_1})+G(\overline{s_2})+G(\overline{s_3})=60+25+70=155$,
and
\begin{align*}
	f(4)&= H(\overline{s_1})G(\overline{s_1})+H(\overline{s_2})G(\overline{s_2})+H(\overline{s_3})G(\overline{s_3}) \cr
	&=3\cdot5\cdot7\cdot9\cdot60
+3\cdot3\cdot7\cdot9\cdot25+3\cdot3\cdot5\cdot9\cdot70\cr
	&=(3\cdot5\cdot7)^2\cdot 9.
\end{align*}
The weighted Genocchi number $f(n)$ is more convenient for us to study, since
it has an explicit simple expression.

\medskip

\begin{thm}\label{th:fn}
	Let $f(n)$ be the weighted Genocchi number defined in \eqref{def:fn}.
	Then,
\begin{equation}
f(n)=(1\cdot 3 \cdot 5 \cdot 7 \cdots (2n-1))^2 \cdot (2n+1)=(2n-1)!!\cdot (2n+1)!!.
\end{equation}
\end{thm}

\begin{proof}
	We successively have
\begin{align*}
f(n)&= \displaystyle\sum_{\bar s} H(\bar s)G(\bar s)  \\
	&= \displaystyle\sum_{\bar s} \displaystyle\frac{H(\bar s) (n+1) T(\bar s)}{2^{2n}} \\
	&= \displaystyle\sum_{\bar s} \displaystyle\frac{(2n+2)! \times \#\bar s}{2^{2n+1}} \\
	&= \displaystyle\frac{(2n+2)!}{2^{2n+1}} \sum_{\bar s} \#\bar s    \\
	&= \displaystyle\frac{(2n+2)!}{2^{2n+1}}\cdot \#\mathcal{S}(2n+1).
\end{align*}

While $\#\mathcal{S}(2n+1)$ equals to the Catalan number $C_n$, we can calculate
that
\begin{align*}
f(n)&= \frac{(2n+2)!}{2^{2n+1}}\cdot C_n  \\
	&= \frac{(2n+2)!}{2^{2n+1}}\cdot \frac{1}{n+1}
	\binom{2n}{n} \\
	&= (2n-1)!!\cdot (2n+1)!!.\qedhere
\end{align*}
\end{proof}

From  Theorem \ref{th:fn}, the weighted Genocchi number $f(n)$ is an odd number.
Therefore, the normal Genocchi number $G_{2n+2}$ is also odd.
This achieves the proof of Theorem~\ref{th:tan}.

\section{Generalizations to $k$-ary trees} \label{sec:kary} 

\medskip

In this section we assume that $k\geq 2$ is an integer.

Recall the {\it hook length formula} for binary trees
described in Section 2.
For general rooted trees $t$ (see \cite[\S5.1.4, Ex. 20]{Knuth1998Vol3}),
we also have

\begin{equation}
	\#\mathcal{L}(t)=\frac{n!}{\prod_{v\in t} h_v},
\end{equation}
where $\mathcal{L}(t)$ denote the set of all
{\it increasing labelled trees} of shape $t$.

Let $L_{kn+1}$ be the number of increasing labelled complete $k$-ary trees with $kn+1$ vertices. Then,
\begin{align}
	L_{kn+1}=\sum_{n_1+\cdots+n_k=n-1}\binom{kn}{kn_1+1, \cdots, kn_k+1}L_{kn_1+1}\cdots L_{kn_k+1}.
\end{align}
Equivalently, the exponential generating function $\phi(x)$ for $L_{kn+1}$
\begin{align*}
\phi(x)=\sum_{n\ge 0}L_{kn+1}\frac{x^{kn+1}}{(kn+1)!}
\end{align*}
is the solution of the differential equation
\begin{equation}\label{eq:phi}
\phi'(x)=1+\phi^k(x)
\end{equation}
such that $\phi(0)=0$.

Let $\psi(x)$ be the exponential generating function for
$M_{k^2 n-kn+k}$ which is defined in Theorem \ref{th:kary},
$$\psi(x):=
\sum_{n\ge 0}M_{k^2 n-kn+k}\frac{x^{k^2n-kn+k}}{(k^2n-kn+k)!}.$$
Then
\begin{equation}\label{eq:psi}
{\psi(x)}=x \cdot \phi\left(\displaystyle \frac{x^{k-1}}{k!}\right).
\end{equation}

From identities \eqref{eq:phi} and \eqref{eq:psi}, Theorem \ref{th:kary} can be restated in the form of power series and differential equations:
\begin{cor}
Let  $\psi(x)$ be a power series satisfying the following differential equation
$$
 x\psi'(x)-\psi(x)=\frac{k-1}{k!}\Bigl(x^k+\psi^k(x)\Bigr),
$$
with $\psi(0)=0$.
Then, for each $n\geq 1$, the coefficient of
	$\displaystyle\frac{x^{k^2n-kn+k}}{(k^2n-kn+k)!}$ in $\psi(x)$ is an integer.
Moreover, it is congruent to

	$(i)$  $1 \pmod k$, if $k=p$;

	$(ii)$ $1 \pmod {p^2}$, if $k=p^t$ with $t\geq 2$;

 $(iii)$ $0 \pmod k$, otherwise.
\end{cor}

\medskip

When $k=2$, $L_{2n+1}$ is just the tangent number $T_{2n+1}$ and $M_{2n+2}$ is the Genocchi number $G_{2n+2}$.
For $k=3$ and $4$, the initial values of $L_{kn+1}$ and $M_{k^2 n-kn+k}$ are
reproduced below:

\medskip

\[\begin{tabular}{c|c|c}
	$n$ & $L_{3n+1}$ & $M_{6n+3}$ \\
\hline
	0 & 1 & 1 \\
	1 & 6 &  70 \\
	2 & 540 & 500500 \\
	3 & 184680 & 43001959000\\
	4 & 157600080 & 21100495466050000 \\
	5 & 270419925600 & 39781831724228093500000
\end{tabular}\]

\smallskip
\centerline{Table for $k=3$}

\medskip

\[\begin{tabular}{c|c|c}
	$n$ & $L_{4n+1}$ & $M_{12n+4}$ \\
\hline
	0 & 1 & 1  \\
	1 & 24 & 525525  \\
	2 & 32256 & 10258577044340625  \\
	3 & 285272064 & 42645955937142729593062265625 \\
	4 & 8967114326016 & 6992644904557760596067178252404694486328125  \\
\end{tabular}\]
\centerline{Table for $k=4$}

\medskip

\medskip

Now we define an equivalence relation ({\it $k$-pivoting}) on the set of all (unlabelled) complete $k$-ary trees $\mathcal{R}(kn+1)$.
	A {\it basic $k$-pivoting}  is a rearrangement of
	the $k$ subtrees of a non-leaf vertex $v$.
%
Let $r_1$, $r_2$ be two complete $k$-ary trees, if $r_1$ can be changed to $r_2$ by a finite sequence of basic $k$-pivotings, we write $r_1\sim r_2$.
Hence the set of all complete $k$-ary trees can be partitioned into several equivalence classes. Let $\mathcal{\bar R}(kn+1) = \mathcal{R}(kn+1)/\!\!\sim$, define $\#\mathcal{L}(\bar r) = \#{\mathcal{L}}(r)$ for $r \in\bar r$,
then we have
\begin{equation}\label{eq:knormaltree}
 \sum_{\bar r\in\mathcal{\bar R}(kn+1)} L(\bar r)= L_{kn+1},
\end{equation}
where
\begin{equation}
L(\bar r)=\sum_{r\in \bar r} \# \mathcal{L}(r) = \#\bar r\times \#\mathcal{L}(\bar r).
\end{equation}

\medskip

Similar to the case of the tangent numbers,
this equivalence relation implies that $L(\bar r)$ is divisible by $(k!)^n$.
There is still a stronger divisibility, stated as below:

\begin{lem}\label{lem:kdivisibility}
	For each $\bar r\in \mathcal{\bar R}(kn+1)$,
the number $(k^2 n-kn+k)!L(\bar r)/(kn+1)!$ is divisible by $(k!)^{kn+1}$.
\end{lem}

\begin{proof}
First, we show that the coefficient $(k^2 n-kn+k)!/(kn+1)!$ is divisible by $(k-1)!^{kn+1}$. In fact,
	\begin{equation}\label{eq:divk-1}
		\displaystyle\frac{(k^2 n-kn+k)!}{(kn+1)!\cdot (k-1)!^{kn+1}}
		=  (k^2n-kn+k)\cdot \displaystyle\prod_{i=1}^{kn+1}\binom{i(k-1)-1}{k-2}.
\end{equation}
	It remains to prove
\begin{equation}\label{eq:kdivisibility}
	k^{kn+1}\mid \frac{(k^2 n-kn+k)!\ L(\bar r)}{(kn+1)!\cdot (k-1)!^{kn+1}}.
\end{equation}

For each vertex $v$ in a complete $k$-ary tree $r$,
we observe that the hook length $h_v$ satisfies $h_v\equiv 1\pmod k$.
Thus,
\begin{align*}
H(\bar r)=\prod_{v\in r}h_v\equiv 1\pmod k.
\end{align*}
Consequently, relation (\ref{eq:kdivisibility}) is equivalent to
	\begin{align*}
	k^{kn+1}\mid \frac {(k^2 n-kn+k)!\ L(\bar r)H(\bar r)}{(kn+1)!\cdot (k-1)!^{kn+1}},
\end{align*}
which can be rewritten  as
	\begin{align}\label{eq:divk!}
		(k!)^{kn+1}\mid (k^2 n-kn+k)! \times \frac{L(\bar r)H(\bar r)} {(kn+1)!}.
\end{align}

We will prove this divisibility using the following combinatorial model.
Let $r$ be a complete $k$-ary tree with $kn+1$ vertices. It is easy to show that $r$ has $(k-1)n+1$ leaves.
Replacing all leaves of $r$ by the complete $k$-ary tree with $k+1$ vertices, we get a new tree with $k^2 n-kn+k$ leaves, denoted by $r^+$.
	Let $\mathcal{L}^+(r^+)$ be the set of all leaf-labelled tree of shape $r^+$, 
	obtained from $r^+$ by labeling all the leaves
with ${1,2,\ldots, k^2 n-kn+k}$.
It is clear that $\#\mathcal{L}^+(r^+)=(k^2 n-kn+k)!$.
On the other hand,
by the hook length formula we have
\begin{equation*}
		\displaystyle\frac{L(\bar r)H(\bar r)}{(kn+1)!}
	= \displaystyle\frac{H(\bar r)\times\#\bar r\times \#\mathcal{L}(r)}{(kn+1)!}
    = \#\bar r.
\end{equation*}
Thus,
	the right-hand side of \eqref{eq:divk!} is equal to
	$(k^2 n-kn+k)! \times \#\bar r $, that is,
the number of all leaf-labelled trees of shape $r^+$ such that $r\in\bar r$.

\medskip

\medskip

Translate the $k$-pivoting to the set of all leaf-labelled trees of shape $r^+$ such that $r\in\bar r$. It is easy to check that the $k$-pivoting is still an
equivalence relation.
Since a leaf-labelled tree has $kn+1$ non-leaf vertices, there are $(k!)^{kn+1}$ leaf-labelled trees in each equivalence class, which
	implies that
	the right-hand side of \eqref{eq:divk!}
	is divisible by $(k!)^{kn+1}$.
\end{proof}

The following two lemmas will be used for proving Theorem \ref{th:kary}.

\begin{lem}[Legendre's formula]\label{th:Legendre}
Suppose that $p$ is prime number.
For each positive integer $k$,
let $\alpha(k)$ be the highest power of $p$ dividing $k!$
and $\beta(k)$ be the sum of all digits of $k$ in base $p$.
	Then,
\begin{align}\label{eq:Legendre}
\alpha(k)=\sum_{i\ge 1}\left\lfloor\frac{k}{p^i}\right\rfloor
	=\frac {k-\beta(k)}{p-1}.
\end{align}
\end{lem}

For the proof of Lemma \ref{th:Legendre}, see \cite[p. 263]{Dickson1919}.

\begin{lem}\label{th:p2}
Let $p\ge 3$ be a prime number, then
\begin{align}\label{eq:p2}
(pk+1)(pk+2)\cdots(pk+p-1)\equiv (p-1)! \pmod{p^2}.
\end{align}
\end{lem}

\begin{proof}
The left-hand side of \eqref{eq:p2} is equal to
\begin{align*}
(pk)^{p-1}e_0+\cdots+(pk)^2e_{p-3}+(pk)e_{p-2}+e_{p-1}\equiv (pk)e_{p-2}+(p-1)!\pmod{p^2},
\end{align*}
	where $e_j:=e_j(1, 2, \cdots, p-1)$ are the elementary symmetric functions. See \cite{Macdonald1995}.
Since
\begin{align*}
e_{p-2}=(p-1)!\displaystyle\sum_{i}i^{-1}\equiv (p-1)!\sum_{i}i\equiv(p-1)!\frac{p(p-1)}2\equiv 0\pmod p,
\end{align*}
equality	\eqref{eq:p2} is true.
\end{proof}

\goodbreak

We are ready to prove Theorem \ref{th:kary}.

\begin{proof}[Proof of Theorem \ref{th:kary}]
The first part (a) is an immediate consequence
of Lemma \ref{lem:kdivisibility} and (\ref{eq:knormaltree}).
Let $n\ge 1$,
we construct the following weighted function
	\begin{equation*}
f(n)=\sum_{\bar r\in\mathcal{\bar R}(kn+1)}H(\bar r)M(\bar r),
\end{equation*}
where
\begin{align*}
M(\bar r)=\displaystyle\frac{(k^2 n-kn+k)!\, L(\bar r)}{(k!)^{kn+1}\, (kn+1)!}.
\end{align*}
Since $H(\bar r)\equiv 1\pmod k$, we have
	\begin{equation}\label{eq:modfn}
	f(n)\equiv \sum_{\bar r\in\mathcal{\bar R}(kn+1)}M(\bar r) = M_{k^2 n-kn+k}
	\pmod k.
\end{equation}
Thus, we only need to calculate $f(n)$.
\begin{align*}
f(n)&= \displaystyle\sum_{\bar r} H(\bar r)M(\bar r)  \\
	&= \displaystyle\sum_{\bar r} \displaystyle\frac{H(\bar r) \times (k^2 n-kn+k)!\, L(\bar r)}{(k!)^{kn+1}\, (kn+1)!} \\
	&= \displaystyle\sum_{\bar r} \displaystyle\frac{(k^2 n-kn+k)! \times \#\bar r}{(k!)^{kn+1}} \\
	&= \displaystyle\frac{(k^2 n-kn+k)!}{(k!)^{kn+1}} C_k(n),
\end{align*}
	where  $C_k(n)$ is the number of all (unlabelled) complete $k$-ary trees, that is equal to the Fuss-Catalan number \cite{Aval2008}
\begin{align*}
	C_k(n)=\displaystyle\frac{(kn)!}{n!(kn-n+1)!}.
\end{align*}
Consequently,
\begin{align}
f(n)
	&= \displaystyle\frac{(k^2 n-kn+k)!}{(k!)^{kn-n+1}(kn-n+1)!}\cdot\frac{(kn)!}{(k!)^nn!} \label{eq:fn1} \\
		&= \displaystyle\prod_{i=0}^{kn-n}\binom{ik+k-1}{k-1}\times \displaystyle \prod_{j=0}^{n-1}\binom{jk+k-1}{k-1}.\label{eq:fn2}
\end{align}

For proving the second part (b), there are three cases to be considered
depending on the value of $k$.

(b1) $k=p$ is a prime integer. We have
	\begin{equation*}
		\binom{ip+p-1}{p-1}
 =\frac{(ip+1)(ip+2)\cdots(ip+p-1)}{1\times 2\times \cdots \times (p-1)}  \equiv 1\pmod{p}.
\end{equation*}
	Thus $f(n)\equiv 1 \pmod p$ by identity \eqref{eq:fn2}.

	\smallskip

(b2)
$k=p^t \ (t\ge 2)$ where $p$ is a prime integer.
	If $p\ge 3$,
by Lemma \ref{th:p2}, we have
\begin{align*}
	\binom{ip^t+p^t-1}{p^t-1} &=\prod_{s=0}^{p^{t-1}-1}\frac{(ip^t+sp+1)\cdots(ip^t+sp+p-1)}{(sp+1)\cdots(sp+p-1)}\cdot \prod_{s=1}^{p^{t-1}-1}\frac{ip^t+sp}{sp} \\
	& \equiv \left[\frac{(p-1)!}{(p-1)!}\right]^{p^{t-1}}\cdot \binom{ip^{t-1}+p^{t-1}-1}{p^{t-1}-1} \pmod{p^2} \\
	& \equiv \binom{ip^{t-1}+p^{t-1}-1}{p^{t-1}-1}  \pmod{p^2}\\
	& \equiv \cdots\\
	& \equiv \binom{ip+p-1}{p-1} \pmod{p^2}\\
& = \frac{(ip+1)(ip+2)\cdots(ip+p-1)}{1\times 2\times \cdots \times (p-1)} \\
& \equiv 1\pmod{p^2}.
\end{align*}

Thus $f(n)\equiv 1\pmod {p^2}$ for $k=p^t$ with $p\geq 3$ and $t\geq 2$.

Now suppose $p=2$ and $k=2^t$ ($t\geq 2$). We have
\begin{align*}
	\binom{i2^t+2^t-1}{2^t-1} &=\prod_{s=0}^{2^{t-1}-1}\frac{i\cdot 2^t+2s+1}{2s+1}\cdot \prod_{s=1}^{2^{t-1}-1}\frac{i\cdot 2^t+2s}{2s} \\
& = \prod_{s=0}^{2^{t-2}-1}\frac{(i\cdot 2^t+4s+1)(i\cdot 2^t +4s+3)}{(4s+1)(4s+3)}\cdot \prod_{s=1}^{2^{t-1}-1}\frac{i\cdot 2^{t-1}+s}{s} \\
	& \equiv \left(\frac{-1}{-1}\right)^{2^{t-2}}\cdot \binom{i\cdot 2^{t-1}+2^{t-1}-1}{2^{t-1}-1}  \pmod{4}\\
	& \equiv\binom{i\cdot 2^{t-1}+2^{t-1}-1}{2^{t-1}-1}\pmod{4}\\
	& \equiv \cdots \\
	& \equiv\binom{i\cdot 2+2-1}{2-1}\pmod{4}\\
& = 2i+1.
\end{align*}
Therefore, by identity \eqref{eq:fn2}, we can check that
\begin{align*}
	f(n)\equiv \prod_{i=0}^{(2^t-1)n}(2i+1)\times \prod_{j=0}^{n-1}(2j+1) \equiv 1 \pmod 4.
\end{align*}

(b3) Suppose that $k$ has more than one prime factors. We want to prove $f(n)\equiv 0\pmod k$.
Let $p$ be a prime factor of $k$,  and write $k=b p^m$ with $b\geq 2$ and $p\nmid b$.
Notice that 
$f(n)\mid f(n+1)$ by identity \eqref{eq:fn1}. Thus, it suffices to show that
\begin{equation}\label{eq:f1}
	f(1)=	\frac{(k^2)!}{(k!)^{k+1}} \equiv 0\pmod {p^m},
\end{equation}
which is equivalent to
\begin{equation}\label{eq:f1'}
	\alpha(b^2 p^{2m}) - (bp^m+1)\, \alpha(b p^m) \geq m.
\end{equation}
By Legendre's formula \eqref{eq:Legendre},
the left-hand side of \eqref{eq:f1'}  is equal to
\begin{align*}
	\Delta
	&= \frac 1{p-1} \Bigl( b^2 p^{2m} - \beta(b^2) -(b p^m+1) (bp^m -\beta(b)) \Bigr) \\
	&= \frac 1{p-1} \Bigl( \beta(b) - \beta(b^2) +b p^m \beta(b) -b p^m \Bigr).
\end{align*}
Since $\beta(b^2) \leq b \beta(b)$ and $\beta(b)\geq 2,\, b\geq 2$, we have
\begin{align}
	\Delta
	&\geq  \frac 1{p-1} \Bigl( (bp^m -b +1)\beta(b)  -b p^m \Bigr)\nonumber\\
	&\geq  \frac 1{p-1} \Bigl( b(p^m -2) +2  \Bigr)\nonumber\\
	&\geq  \frac 1{p-1} \Bigl( 2p^m -2 \Bigr)\nonumber\\
	&\geq m.
\end{align}
This completes the proof.
\end{proof}


\smallskip

{\bf Acknowledgments}. The first author would like to thank Zhi-Ying Wen for inviting me to Tsinghua University where the paper was finalized.


\end{document}